\documentclass[12pt]{article}
\usepackage{amsmath,amssymb,amsthm}
\usepackage[hidelinks]{hyperref}
\usepackage{cleveref}

\usepackage{enumitem}
\newlist{steps}{enumerate}{1}
\setlist[steps, 1]{label = Step \arabic*:}

\newtheorem{theorem}{Theorem}
\newtheorem{proposition}[theorem]{Proposition}
\newtheorem{lemma}[theorem]{Lemma}
\newtheorem{corollary}[theorem]{Corollary}

\theoremstyle{definition}

\theoremstyle{remark}

\crefname{equation}{}{}
\Crefname{equation}{Equation}{Equations}
\crefname{theorem}{Theorem}{Theorems}
\Crefname{theorem}{Theorem}{Theorems}
\crefname{lemma}{Lemma}{Lemmas}
\Crefname{lemma}{Lemma}{Lemmas}
\crefname{proposition}{Proposition}{Propositions}
\Crefname{proposition}{Proposition}{Propositions}
\crefname{corollary}{Corollary}{Corollaries}
\Crefname{corollary}{Corollary}{Corollaries}
\crefname{conjecture}{Conjecture}{Conjectures}
\Crefname{conjecture}{Conjecture}{Conjectures}
\crefname{section}{Section}{Sections}
\Crefname{section}{Section}{Sections}
\crefname{example}{Example}{Examples}
\Crefname{example}{Example}{Examples}
\crefname{problem}{Problem}{Problems}
\Crefname{problem}{Problem}{Problems}
\crefname{remark}{Remark}{Remarks}
\Crefname{remark}{Remark}{Remarks}
\crefname{figure}{Figure}{Figures}
\Crefname{figure}{Figure}{Figures}

\newcommand{\doi}[1]{\href{http://dx.doi.org/#1}{\texttt{doi:#1}}}

\title{Tur\'{a}n numbers $T(n,5,3)$ and \\ graphs without induced $5$-cycles}

\author{
Iliya Bluskov 
\thanks{
P.O. Box 33031, West Vancouver, B.C. V7V 4W7, Canada,
E-mail: {\tt
lotbook@telus.net}.} 
\and
Jan de Heer 
\thanks{
Tetterode 1,
2151RC Nieuw Vennep,
Netherlands,
E-mail: {\tt
deheertjes@hetnet.nl}.} 
\and
Alexander Sidorenko\thanks{
R\'{e}nyi Institute, Budapest, Hungary,
E-mail: {\tt
sidorenko.ny@gmail.com}
} 
}

\date{\today}

\begin{document}

\maketitle

\begin{abstract}

Tur\'{a}n number $T(n,5,3)$ is the minimum size of a system of triples 
out of a base set $X$ of $n$ elements such that 
every quintuple in $X$ contains a triple from the system. 
The exact values of $T(n,5,3)$ are known for $n \leq 17$. 
Tur\'{a}n conjectured that $T(2m,5,3) = 2\binom{m}{3}$, 
and no counterexamples have been found so far. 
If this conjecture is true, then 
$T(2m+1,5,3) \geq \lceil m(m-2)(2m+1)/6\rceil$. 
We prove the matching upper bound for all $n = 2m+1 > 17$ except $n=27$. 
\end{abstract}


\section{Introduction}

Let $X$ be an $n$-element set and $n \geq k \geq r$. 
We call a system $S$ of $r$-element subsets of $X$ 
a \emph{Tur\'{a}n} $(n,k,r)$-system 
if any $k$-element subset of $X$ 
contains at least one subset from $S$. 
\emph{Tur\'{a}n number} $T(n,k,r)$ 
is the smallest size of such a system. 
Mantel \cite{Mantel:1907} (for $k=3$) 
and Tur\'{a}n \cite{Turan:1941} (for all $k$) 
found the exact values of $T(n,k,2)$. 
More information can be found in the survey \cite{Sidorenko:1995}. 

If we partition $X$ 
into almost equal parts $X'$ and $X''$ 
and take all triples within each part, 
then by the pigeonhole principle, 
we get a Tur\'{a}n $(n,5,3)$-system. 
Hence, 
\begin{equation}\label{eq:simple}
  T(n,5,3) \:\leq \binom{\lceil n/2 \rceil}{3} + 
                  \binom{\lfloor n/2 \rfloor}{3} \, .
\end{equation}
Tur\'{a}n conjectured that for even $n$, 
inequality \cref{eq:simple} becomes equality 
(see \cite{Erdos:1977,Turan:1970}).  
No counterexample has been found so far, 
and the conjecture has been verified for all even $n \leq 16$. 
The exact values of $T(n,5,3)$ and all extremal Tur\'{a}n systems 
(up to isomorphism) have been determined for $n \leq 17$
(see \cite{Boyer:1994,Markstrom:2022}): 

\vspace{3mm}
\begin{tabular}{cccccccccccccc}
  $n$           
  & 5 & 6 & 7 & 8 &  9 & 10 & 11 & 12 & 13 & 14 & 15 &  16 &  17
    \\ 
  $T(n,5,3)$    
  & 1 & 2 & 5 & 8 & 12 & 20 & 29 & 40 & 52 & 70 & 89 & 112 & 136
\end{tabular}
\vspace{3mm}

For odd $n \geq 9$, inequality \cref{eq:simple} is not tight. 
If Tur\'{a}n's conjecture is true, 
that is, $T(2m,5,3) = 2 \binom{m}{3}$, 
then by using the the inequality 
$T(n,k,r) \geq \frac{n}{n-r} \, T(n-1,k,r)$ from \cite{Katona:1964}, 
we get 
\begin{align}\label{eq:lower_bound}
  \nonumber
  T(2m+1,5,3) & \geq \left\lceil\frac{m(m-2)(2m+1)}{6} \right\rceil
  = \\ & =
  \begin{cases}
      \;\;\;\;\: \frac{1}{6} (2m+1)m(m-2) \;\;\;\;
        \;\;{\rm if}\;\; m \;\;{\rm is \; even}; \\
      \frac{1}{6} (2m-3)(m-1)(m+1) 
        \;\;{\rm if}\;\; m \;\;{\rm is \; odd}.
  \end{cases}
\end{align}

A construction based on the finite affine geometry AG(2,3) 
(see \cite{Boyer:1994,Sidorenko:1995}) 
provides an upper bound that matches \cref{eq:lower_bound} 
for $n=2m+1 \equiv 1 \!\!\pmod{8}$.
In this article, we describe new constructions that are 
based on triangles and independent sets of size $3$ 
in graphs without induced $5$-cycles. 
These constructions provide the matching upper bound 
for all $n=2m+1 > 17$ except $n=27$.

The article is organized as follows.
A general description of our constructions is given in \cref{sec:Goodman}. 
The case $n \equiv 1 \pmod{4}$ 
is considered in details in \cref{sec:regular}, 
and the case $n \equiv 3 \pmod{4}$ in \cref{sec:almost}. 
Concluding remarks are given in \cref{sec:remarks}.

\section{Triangles and anti-triangles in a graph}\label{sec:Goodman}

An \emph{anti-triangle} in a graph is a set of three independent vertices.
The $5$-cycle is the only $5$-vertex graph 
that contains neither a triangle nor an anti-triangle. 
This yields the following observation.

\begin{proposition}
If an $n$-vertex graph $G$ 
does not contain a $5$-cycle as an induced subgraph, 
then the triangles and anti-triangles of $G$ 
form a Tur\'{a}n $(n,5,3)$-system. 
\end{proposition}

In 1959, Goodman \cite{Goodman:1959} determined 
the minimum number of triangles and anti-triangles 
in an $n$-vertex graphs and characterized the extremal graphs.
A graph is called $d$-\emph{regular} if every vertex has degree $d$. 
We call a graph \emph{almost} $d$-\emph{regular} 
if all vertices except one have degree $d$ 
and the remaining vertex has degree $d \pm 1$. 

\begin{theorem}[\cite{Goodman:1959}]
The minimum number of triangles and anti-triangles 
in an $n$-vertex graph $G$ is 
\begin{equation*}
  M(n) = 
  \begin{cases}
      \;\;\;\;\;\;\;\;\;\; \;\;\;\;
      2\binom{m}{3}\;\;\;\;\;\;\;\;\;\;\;\;\;\;\;\;\;\;\;\;
      {\rm if}\;\; n=2m; \\
      \;\;\;\;\: \frac{1}{6} (2m+1)m(m-2) \;\;\;\;
        \;\;{\rm if}\;\; n=2m+1,\;\; m \;{\rm is \; even}; \\
      \frac{1}{6} (2m-3)(m-1)(m+1) 
        \;\;{\rm if}\;\; n=2m+1,\;\; m \;{\rm is \; odd}.
  \end{cases}
\end{equation*}
The total number of triangles and anti-triangles is equal to $M(n)$ 
if and only if 
all vertices of $G$ have degrees $\frac{n}{2}$ and $\frac{n}{2}-1$ $($for even $n)$, 
or $G$ is $\frac{n-1}{2}$-regular $($for $n \equiv 1 \pmod{4})$,
or $G$ is almost $\frac{n-1}{2}$-regular $($for $n \equiv 3 \pmod{4})$.
\end{theorem}

The right hand side of \cref{eq:simple} for $n=2m$ 
and the right hand side of \cref{eq:lower_bound} for $n=2m+1$
are both equal to $M(n)$. 
It turns out that 
for each $n \neq 13$, 
there exists an $n$-vertex graph $G$ without induced $5$-cycles 
such that the number of triangles and anti-triangles in $G$ 
provides the best known upper bound for $T(n,5,3)$. 
For instance, if $n=2m$, two disjoint copies of $K_m$ 
contain $2\binom{m}{3}$ triangles, 
no anti-triangles, 
and no induced $5$-cycles. 
The case of odd $n$ is more complicated. 
Our main result is 

\begin{theorem}\label{th:main}
For every odd $n \geq 17$, with the exception of $n=27$, 
there exists an $\frac{n-1}{2}$-regular $($if $n \equiv 1 \pmod{4})$
or an almost $\frac{n-1}{2}$-regular $($if $n \equiv 3 \pmod{4})$
$n$-vertex graph without induced $5$-cycles. 
For $n=27$, there is 
a $27$-vertex graph without induced $5$-cycles 
with the total number of triangles and anti-triangles 
equal to $M(27)+1$. 
\end{theorem}

\begin{corollary}
$T(n,5,3) \leq M(n)$ for $n \geq 17$, $n \neq 27$, and 
$T(27,5,3) \leq M(27)+1 = 645$. 
\end{corollary}

As we will show in \cref{sec:regular,sec:almost}, 
it is quite easy to construct 
an (almost) $d$-regular $n$-vertex graph without induced $5$-cycles 
as long as $d \neq \frac{n-1}{2}$. 
The idea of the proof of \cref{th:main} 
is to assemble an (almost) $\frac{n-1}{2}$-regular $n$-vertex graph 
without induced $5$-cycles 
out of smaller (almost) regular graphs of various degrees.

For disjoint graphs $G_1,G_2,\ldots,G_9$, 
we define their \emph{grid graph} 
$\left[\begin{smallmatrix}
  G_1 & G_2 & G_3 \\
  G_4 & G_5 & G_6 \\
  G_7 & G_8 & G_9 
\end{smallmatrix}\right]$ 
as the union of $G_1,G_2,\ldots,G_9$ with addition of edges 
between the vertices of graphs that appear 
in the same row or in the same column. 
It is easy to see that if $G_1,G_2,\ldots,G_9$ 
do not contain induced $5$-cycles, 
then the grid graph does not contain such cycles either. 
In the case when $G_i$ is a $d_i$-regular graph with $n_i$ vertices, 
we will also denote the grid graph by 
\[\begin{bmatrix}
  (n_1,d_1) & (n_2,d_2) & (n_3,d_3) \\
  (n_4,d_4) & (n_5,d_5) & (n_6,d_6) \\
  (n_7,d_7) & (n_8,d_8) & (n_9,d_9)
\end{bmatrix}.\] 
This graph is $n$-vertex and $\frac{n-1}{2}$-regular if and only if 
({\it a}) $n_1+n_2+\ldots+n_9=n$, and 
({\it (b} )for each $i=1,2,\ldots,9$,
the sum of $n_j$ with $j \neq i$ from the same row and column plus $d_i$ 
is equal to $\frac{n-1}{2}$. 
In particular, 
$n_2+n_3+n_4+n_7+d_1 = \frac{n-1}{2}$,
$n_1+n_3+n_5+n_8+d_2 = \frac{n-1}{2}$, and so on.
When trying to find suitable $n_i,d_i$, 
we have $9 \cdot 2 = 18$ variables 
to satisfy $1+9=10$ equations. 
The system of these equations has a rational valued solution: 
$n_i = \frac{n}{9}$, $d_i = \frac{n-9}{18}$ ($i=1,2,\ldots,9$). 
We will be looking for an integer valued solution 
such that the component graphs with parameters $(n_i,d_i)$ exist.

\section{Regular graphs without induced 5-cycles}\label{sec:regular}

In this section, we will prove \cref{th:main} for $n \equiv 1 \pmod{4}$.

\begin{lemma}\label{th:reg_even}
For any even $n$ and any $0 \leq d \leq n-1$, 
there exists a $d$-regular graph with $n$ vertices 
that does not contain $5$-cycles as induced subgraphs. 
\end{lemma}

\begin{proof}[\bf{Proof}]
We will use induction on $n$. 
The base case $n=2$ is trivial. 
Let $n=2m \geq 4$. 
If $d=m-1$, the union of two disjoint cliques of size $m$ 
gives the graph we seek.
Suppose $0 \leq d \leq m-2$. 
Let $r$ be the even number from $\{d+1,d+2\}$. 
Then $n-r \geq 2m - (d+2) \geq m > d$. 
By the induction hypothesis, 
there exists a $d$-regular graph $G_{r,d}$ with $r$ vertices 
and a $d$-regular graph $G_{n-r,d}$ with $n-r$ vertices 
that do not contain $5$-cycles as induced subgraphs. 
Then the union of disjoint copies of $G_{r,d}$ and $G_{n-r,d}$ 
is the graph we seek. 
If $d\geq m$, set $d' := (n-1)-d$. Then $0 \leq d' \leq m-1$. 
We proved above that there exists 
an $n$-vertex $d'$-regular graph without induced $5$-cycles. 
Its complement is a $d$-regular graph without induced $5$-cycles.
\end{proof}

\begin{lemma}\label{th:reg_odd}
For any odd $n$ and any even $d$ with 
$0 \leq d \leq n-1$ and $d \neq \frac{n-1}{2}$, 
there exists a $d$-regular graph with $n$ vertices 
that does not contain $5$-cycles as induced subgraphs. 
\end{lemma}

\begin{proof}[\bf{Proof}]
Let $n=2m+1$, so $d \neq m$. 
Suppose $0 \leq d \leq m-1$. 
By \cref{th:reg_even}, 
there exists a $d$-regular graph $G_{n-d-1,d}$ with $n-d-1$ vertices 
that does not contain $5$-cycles as an induced subgraphs. 
Then the union of $G_{n-d-1,d}$ and $K_{d+1}$
is an $n$-vertex $d$-regular graph without induced $5$-cycles. 
Now suppose $d>m-1$. As $d \neq m$, we have $d \geq m+1$. 
Set $d' := (n-1)-d$. Then $d'$ is even and $0 \leq d' \leq m-1$. 
We proved above that there exists 
an $n$-vertex $d'$-regular graph without induced $5$-cycles. 
Its complement is a $d$-regular graph without induced $5$-cycles.
\end{proof}

\begin{lemma}\label{th:4r}
For every $m \equiv 0 \pmod{4}$, 
there exists an $m$-regular graph with $2m+1$ vertices 
that does not contain $5$-cycles as induced subgraphs. 
\end{lemma}

\begin{proof}[\bf{Proof}]
Let $m=4r$. The graph given by
\[\begin{bmatrix}
  (1,0) & (r,r-1) & (r,r-1) \\
  (r,r-1) & (r,0) & (r,0) \\
  (r,r-1) & (r,0) & (r,0) 
\end{bmatrix}\] 
is $(4r)$-regular, on $8r+1$ vertices and has no induced $5$-cycles. 
\end{proof}

\begin{lemma}\label{th:4r+2}
For every $m \equiv 2 \pmod{4}$ with $m \geq 10$,  
there exists an $m$-regular graph with $2m+1$ vertices 
that does not contain $5$-cycles as induced subgraphs. 
\end{lemma}

\begin{proof}[\bf{Proof}]
We will split the proof into $7$ separate cases. 
In each case, we will construct the desired graph 
as the grid graph of regular graphs (without induced $5$-cycles) 
whose existence is guaranteed by \cref{th:reg_even,th:reg_odd,th:4r}. 

\medskip

{\it Case 1.} $m=6r-2$. (This covers cases $m \equiv 10,22,34 \pmod{36}$.) 
The graph given by
\[\begin{bmatrix}
  (2r-1,0) & (2r-1,0) & (2r-1,0) \\
  (r,r-1)  & (r,r-1)  & (r,r-1)  \\
  (r,r-1)  & (r,r-1)  & (r,r-1) 
\end{bmatrix}\] 
is $(6r-2)$-regular, on $12r-3$ vertices and has no induced $5$-cycles. 

\medskip

{\it Case 2.} $m=36r+2$ with $r \geq 1$. The graph given by 
\[\begin{bmatrix}
  (8r+1,4r+2) & (8r+1,4r-2) & (8r+1,4r-2) \\
  (8r-1,4r)   & (8r+1,4r)   & (8r+1,4r)   \\
  (8r-1,4r)   & (8r+1,4r)   & (8r+1,4r) 
\end{bmatrix}\] 
is $(36r+2)$-regular, on $72r+5$ vertices and has no induced $5$-cycles. 

\medskip

{\it Case 3.} $m=36r+6$ with $r \geq 1$. The graph given by 
\[\begin{bmatrix}
  (8r+2,4r+3) & (8r+2,4r-2) & (8r+2,4r-2) \\
  (8r  ,4r+1) & (8r+2,4r)   & (8r+2,4r)   \\
  (8r-1,4r)   & (8r+2,4r+1) & (8r+2,4r+1) 
\end{bmatrix}\] 
is $(36r+6)$-regular, on $72r+13$ vertices and has no induced $5$-cycles. 

\medskip

{\it Case 4.} $m=36r+14$. The graph given by
\[\begin{bmatrix}
  (8r+4,4r+2) & (8r+4,4r  ) & (8r+3,4r  ) \\
  (8r+3,4r+2) & (8r+3,4r  ) & (8r+3,4r+2) \\
  (8r+2,4r  ) & (8r+4,4r+2) & (8r+3,4r+2) 
\end{bmatrix}\] 
is $(36r+14)$-regular, on $72r+29$ vertices and has no induced $5$-cycles. 

\medskip

{\it Case 5.} $m=36r+18$. The graph given by
\[\begin{bmatrix}
  (8r+6,4r+2) & (8r+4,4r)   & (8r+4,4r+1) \\
  (8r+4,4r)   & (8r+4,4r+2) & (8r+4,4r+3) \\
  (8r+4,4r+1) & (8r+4,4r+3) & (8r+3,4r+2) 
\end{bmatrix}\] 
is $(36r+18)$-regular, on $72r+37$ vertices and has no induced $5$-cycles. 

\medskip

{\it Case 6.} $m=36r+26$. The graph given by
\[\begin{bmatrix}
  (8r+7,4r+2) & (8r+6,4r+1) & (8r+6,4r+3) \\
  (8r+6,4r+1) & (8r+6,4r+2) & (8r+6,4r+4) \\
  (8r+6,4r+3) & (8r+6,4r+4) & (8r+4,4r+2) 
\end{bmatrix}\] 
is $(36r+26)$-regular, on $72r+53$ vertices and has no induced $5$-cycles. 

\medskip

{\it Case 7.} $m=36r+30$. The graph given by
\[\begin{bmatrix}
  (8r+8,4r+4) & (8r+7,4r+2) & (8r+6,4r+2) \\
  (8r+7,4r+2) & (8r+7,4r+2) & (8r+7,4r+4) \\
  (8r+6,4r+2) & (8r+7,4r+4) & (8r+6,4r+4) 
\end{bmatrix}\] 
is $(36r+30)$-regular, on $72r+61$ vertices and has no induced $5$-cycles. 
\end{proof}

\section{Almost regular graphs without induced \\ 5-cycles}\label{sec:almost}

In this section, we will prove \cref{th:main} for $n \equiv 3 \pmod{4}$.

\begin{lemma}\label{th:almost}
For any $n \equiv 3 \pmod{4}$ and any odd $d$ such that 
$1 \leq d \leq n-1$ and $d \neq \frac{n-1}{2}$, 
there exists an almost $d$-regular graph with $n$ vertices 
that does not contain $5$-cycles as induced subgraphs. 
\end{lemma}

\begin{proof}[\bf{Proof}]
Let $n=4m+3$, so $d \neq 2m+1$. 
Suppose $1 \leq d \leq 2m-1$. 
By \cref{th:reg_even}, 
there exists a $d$-regular graph $G_{n-d-2,d}$ with $n-d-2$ vertices 
that does not contain $5$-cycles as induced subgraphs. 
Let $H_{d+2}$ denote a graph obtained from $K_{d+2}$ by removing $(d+1)/2$ edges 
that do not have common vertices. 
Then $d+1$ vertices in $H_{d+2}$ have degree $d$, 
and one vertex has degree $d+1$.
Since a set of independent edges does not induce a $5$-cycle, 
$H_{d+2}$ does not induce it either. 
Then the union of $G_{n-d-2,d}$ and $H_{d+2}$
is an almost $d$-regular graph with $n$ vertices 
and without induced $5$-cycles. 
Now suppose $d>2m-1$. As $d \neq 2m+1$, we have $d \geq 2m+3$. 
Set $d' := (n-1)-d$. Then $d'$ is odd and $1 \leq d' \leq 2m-1$. 
We proved above that there exists 
an $n$-vertex $d'$-regular graph without induced $5$-cycles. 
Its complement is a $d$-regular graph without induced $5$-cycles.
\end{proof}

\begin{lemma}\label{th:4r+1_or_3}
For every $m \equiv 1 \pmod{2}$ with $m \geq 9$, $m \neq 13$, 
there exists an almost $m$-regular graph with $2m+1$ vertices 
that does not contain $5$-cycles as induced subgraphs. 
\end{lemma}

\begin{proof}[\bf{Proof}]
We will construct the required graph as the grid graph  
$\left[\begin{smallmatrix}
  G_1 & G_2 & G_3 \\
  G_4 & G_5 & G_6 \\
  G_7 & G_8 & G_9 
\end{smallmatrix}\right]$, 
where $G_1$ is an almost regular graph 
without induced $5$-cycles, 
and $G_2,\ldots,G_9$ are regular graphs 
without induced $5$-cycles. 
To indicate that $G_1$ is almost regular, 
we will denote it with double parentheses: 
$((n,d))$ where $n$ is the number of vertices 
and $d$ is the degree of almost all vertices. 
In each case, $G_1,G_2,\ldots,G_9$ will be (almost) regular graphs 
whose existence is guaranteed by 
\cref{th:reg_even,th:reg_odd,th:4r,th:4r+2,th:almost}, 
or a $3$-vertex graph with $1$ edge (which is almost $1$-regular). 

First we will give constructions for $m=9,11,15,17,19,37$. 
They are, respectively, 
\begin{align*}
\begin{bmatrix}
  ((3,1)) & (2,0) & (2,0) \\
   (2,0)  & (2,1) & (2,1) \\
   (2,0)  & (2,1) & (2,1) 
\end{bmatrix}, 
\;\;\;\;
\begin{bmatrix}
  ((3,1)) & (3,0) & (3,0) \\
   (2,1)  & (3,2) & (2,0) \\
   (2,1)  & (2,0) & (3,2) 
\end{bmatrix}, 
\end{align*}\begin{align*}
\begin{bmatrix}
  ((3,1)) & (4,3) & (2,1) \\
   (4,0)  & (4,0) & (4,2) \\
   (4,2)  & (3,0) & (3,2) 
\end{bmatrix}, 
\;\;\;\;
\begin{bmatrix}
  ((3,1)) & (4,2) & (4,2) \\
   (4,2)  & (4,1) & (4,1) \\
   (4,2)  & (4,1) & (4,1) 
\end{bmatrix}, 
\end{align*}\begin{align*}
\begin{bmatrix}
  ((3,1)) & (4,3) & (4,3) \\
   (5,4)  & (4,2) & (3,0) \\
   (5,0)  & (5,0) & (6,2) 
\end{bmatrix}, 
\;\;\;\;
\begin{bmatrix}
  ((7,5)) & (9,4) & ( 9,2) \\
   (6,5)  & (8,4) & ( 9,4) \\
   (8,5)  & (9,2) & (10,2) 
\end{bmatrix}. 
\end{align*}

Next we will consider separately $9$ cases for $m \geq 21$, $m \neq 37$, 
depending on $m$ modulo $18$. 

\medskip

{\it Case 1.} $m=18r+1$ with $r \geq 3$. The graph given by 
\[\begin{bmatrix}
  ((4r-1,2r+3)) & (4r+2,2r-3) & (4r+2,2r-3) \\
   (4r-3,2r+2)  & (4r+2,2r  ) & (4r+1,2r-2) \\
   (4r-3,2r+2)  & (4r+1,2r-2) & (4r+2,2r  )
\end{bmatrix}\] 
is almost $(18r+1)$-regular, on $36r+3$ vertices and has no induced $5$-cycles. 

\medskip

{\it Case 2.} $m=18r+3$ with $r \geq 1$. The graph given by 
\[\begin{bmatrix}
  ((4r+3,2r-1)) & (4r+2,2r-2) & (4r+2,2r-2) \\
   (4r  ,2r  )  & (4r  ,2r+1) & (4r  ,2r+1) \\
   (4r  ,2r  )  & (4r  ,2r+1) & (4r  ,2r+1)
\end{bmatrix}\] 
is almost $(18r+3)$-regular, on $36r+7$ vertices and has no induced $5$-cycles. 

\medskip

{\it Case 3.} $m=18r+5$ with $r \geq 1$. The graph given by 
\[\begin{bmatrix}
  ((4r+3,2r-1)) & (4r+2,2r-2) & (4r+2,2r  ) \\
   (4r+2,2r+2)  & (4r  ,2r-1) & (4r  ,2r+1) \\
   (4r  ,2r-2)  & (4r+2,2r+3) & (4r  ,2r+1)
\end{bmatrix}\] 
is almost $(18r+5)$-regular, on $36r+11$ vertices and has no induced $5$-cycles. 

\medskip

{\it Case 4.} $m=18r+7$ with $r \geq 1$. The graph given by 
\[\begin{bmatrix}
  ((4r+3,2r-1)) & (4r+2,2r-2) & (4r+2,2r+2) \\
   (4r+2,2r  )  & (4r+2,2r+1) & (4r  ,2r+1) \\
   (4r+2,2r  )  & (4r+2,2r+1) & (4r  ,2r+1)
\end{bmatrix}\] 
is almost $(18r+7)$-regular, on $36r+15$ vertices and has no induced $5$-cycles. 

\medskip

{\it Case 5.} $m=18r+9$ with $r \geq 1$. The graph given by 
\[\begin{bmatrix}
  ((4r+3,2r+3)) & (4r+3,2r-2) & (4r+3,2r-2) \\
   (4r  ,2r+1)  & (4r+3,2r+2) & (4r+2,2r  ) \\
   (4r  ,2r+1)  & (4r+2,2r  ) & (4r+3,2r+2)
\end{bmatrix}\] 
is almost $(18r+9)$-regular, on $36r+19$ vertices and has no induced $5$-cycles. 

\medskip

{\it Case 6.} $m=18r+11$ with $r \geq 1$. The graph given by 
\[\begin{bmatrix}
  ((4r+3,2r+3)) & (4r+2,2r+1) & (4r  ,2r+1) \\
   (4r+3,2r-2)  & (4r+4,2r  ) & (4r+3,2r+2) \\
   (4r+3,2r  )  & (4r+3,2r  ) & (4r+2,2r+2)
\end{bmatrix}\] 
is almost $(18r+11)$-regular, on $36r+23$ vertices and has no induced $5$-cycles. 

\medskip

{\it Case 7.} $m=18r+13$ with $r \geq 1$. The graph given by 
\[\begin{bmatrix}
  ((4r+3,2r+3)) & (4r+4,2r-1) & (4r+4,2r-1) \\
   (4r+1,2r+2)  & (4r+4,2r+2) & (4r+3,2r  ) \\
   (4r+1,2r+2)  & (4r+3,2r  ) & (4r+4,2r+2)
\end{bmatrix}\] 
is almost $(18r+13)$-regular, on $36r+27$ vertices and has no induced $5$-cycles. 

\medskip

{\it Case 8.} $m=18r+15$ with $r \geq 1$. The graph given by 
\[\begin{bmatrix}
  ((4r+3,2r+3)) & (4r+4,2r  ) & (4r+4,2r  ) \\
   (4r+2,2r+2)  & (4r+4,2r+1) & (4r+4,2r+1) \\
   (4r+2,2r+2)  & (4r+4,2r+1) & (4r+4,2r+1)
\end{bmatrix}\] 
is almost $(18r+15)$-regular, on $36r+31$ vertices and has no induced $5$-cycles. 

\medskip

{\it Case 9.} $m=18r+17$ with $r \geq 1$. The graph given by 
\[\begin{bmatrix}
  ((4r+3,2r+3)) & (4r+4,2r+1) & (4r+4,2r+1) \\
   (4r+2,2r+2)  & (4r+4,2r+2) & (4r+4,2r+2) \\
   (4r+4,2r+2)  & (4r+5,2r  ) & (4r+5,2r  )
\end{bmatrix}\] 
is almost $(18r+17)$-regular, on $36r+35$ vertices and has no induced $5$-cycles. 
\end{proof}

\begin{lemma}\label{th:27}
There exists a graph with 27 vertices where $24$ vertices have degree $13$, three vertices have degree $12$, and there is no induced $5$-cycle.
\end{lemma}

\begin{proof}[\bf{Proof}]
$\left[\begin{smallmatrix}
  (4,0) & (4,0) & (4,0) \\
  (3,2) & (2,1) & (2,1) \\
  (2,0) & (3,2) & (3,2) 
\end{smallmatrix}\right]$ 
is the required graph. 
The number of triangles and anti-triangles in this graph is 
one higher than $M(27)$. 
\end{proof}

\Cref{th:main} follows from \cref{th:4r,th:4r+2,th:4r+1_or_3,th:27}.

\section{Concluding remarks}\label{sec:remarks}

The case $n=27$ remains unresolved 
and is too big for exhaustive computer search. 
There are three distinct poissibilities:
\begin{itemize}
\item 
$T(27,5,3) = M(27) + 1$ (similar to $n=5,7,11,15)$.
\item 
Every almost 13-regular 27-vertex graph contains an induced 5-cycle,
but (similarly to $n=13$) there exists 
a Tur\'{a}n $(27,5,3)$-system of size $M(27)$ 
which can not be represented by triangles and anti-triangles 
of an almost $13$-regular $27$-vertex graph without induced 5-cycles. 
In this case, $T(27,5,3) = M(27)$.
\item 
There exists an almost 13-regular 27-vertex graph 
without induced 5-cycles, 
so $T(27,5,3) = M(27)$.
\end{itemize}
The first possibility seems more likely.

\medskip

Though the size of the unique extremal Tur\'{a}n $(13,5,3)$-system 
is equal to $M(13)=52$,
its triples can not be represented 
by triangles and anti-triangles of a $6$-regular $13$-vertex graph $G$. 
Indeed, the $52 = 13 \cdot 4$ triples of the Tur\'{a}n system 
are collinear triples of points in the projective geometry PG(2,3). 
If these triples were triangles and anti-triangles in $G$, 
then the four points of each line would induce either a clique 
or an independent set in $G$. 
As the number of lines is odd, 
the number of edges in $G$ and the number of edges in its complement
can not be equal. 

\medskip

When four elements span exactly three triples, 
we call them a $(4,3)$-\emph{configuration}. 
It is easy to see that a system of triangles 
and anti-triangles of a graph 
can not produce a $(4,3)$-configuration. 
The system of collinear triples in PG(2,3) 
does not contain such a configuration either. 
Let $B_k$ be the currently best known upper bound for $T(n,k,3)$, 
so $B_5(n) = M(n)+1$ if $n=5,7,11,15,27$, 
and $B_5(n)=M(n)$ otherwise. 
Then for any $n$, 
there exists a Tur\'{a}n $(n,5,3)$-system of size $B_5(n)$ 
that does not contain a $(4,3)$-configuration. 
The same observation can be made about 
Tur\'{a}n $(n,k,3)$-systems of size $B_k(n)$ for any $k$. 
Razborov \cite{Razborov:2010} and Pikhurko \cite{Pikhurko:2011} 
proved that for sufficiently large $n$, 
the minimum size of a Tur\'{a}n $(n,4,3)$-system 
that does not contain a $(4,3)$-configuration 
is the same as the conjectured value of $T(n,4,3)$. 
Similarly to their results, it would be interesting to prove 
that every Tur\'{a}n $(n,5,3)$-system without a $(4,3)$-configuration 
must have at least $M(n)$ triples.

\medskip

Fon-Der-Flaass \cite{Flaass:1988} used an oriented graph
that does not contain induced directed $4$-cycles 
to construct a Tur\'{a}n $(n,4,3)$-system. 
Our approach is somewhat similar: 
we start with an ordinary graph 
that does not contain induced $5$-cycles 
and use its triangles and anti-triangles 
to construct a Tur\'{a}n $(n,5,3)$-system. 
It is unclear if both approaches could be unified and generalized.

\medskip

There are two self-complementary graphs on $5$ vertices: 
the $5$-cycle and the \emph{bull's head graph} (a triangle with two pendant edges).
Hurkens \cite[p. 379]{Prachatice:1992} 
conjectured that every $\frac{n-1}{2}$-regular $n$-vertex graph 
must contain at least one of them as an induced subgraph.
This conjecture is still open.

\vspace{4mm}
{\bf Acknowledgments.}
No funding was received for conducting this study.
We would like to thank 
Zolt\'{a}n F\"{u}redi and Mikl\'{o}s Simonovits 
for suggesting the references to Tur\'{a}n's conjecture. 
We are also grateful to the anonymous reviewer 
for helping us to improve the presentation of this paper.

\vspace{4mm}
{\bf Conflict of Interests Statement.} 
The authors have no relevant financial or non-financial interests to disclose.

\vspace{4mm}
{\bf Data Accessibility Statement.}
All data analyzed during this study are included in this article.


\begin{thebibliography}{99}

\bibitem{Boyer:1994}
E. D. Boyer, D. L. Kreher, S. P. Radziszowski, and A. Sidorenko: 
\newblock \emph{On $(n,5,3)$-{T}ur\'{a}n systems}, 
\newblock \emph{Ars Combin.}, {\bf 37}(1994), 13--31. 

\bibitem{Erdos:1977}
P. Erd\H{o}s: 
\newblock \emph{Paul {T}ur\'{a}n, 1910--1976: his work in graph theory}, 
\newblock \emph{J. Graph Theory}, {\bf 1}(1977), 97--101,
\doi{10.1002/jgt.3190010204}.

\bibitem{Flaass:1988}
D. G. Fon-Der-Flaas:
\newblock \emph{Method for construction of $(3,4)$-graphs},
\newblock \emph{Math. Notes}, {\bf 44}(1988), 781--783,
\doi{10.1007/BF01158925}.

\bibitem{Goodman:1959}
A. W. Goodman:
\newblock \emph{On sets of acquaintances and strangers at any party},
\newblock \emph{Amer. Math. Monthly}, {\bf 66}(1959), 778--783, 
\doi{10.2307/2310464}.

\bibitem{Katona:1964}
G. Katona, T. Nemetz, and M. Simonovits: 
\newblock \emph{On a graph problem of {T}ur\'{a}n} (in {H}ungarian), 
\newblock \emph{Mat. Lapok}, {\bf 15}(1964), 228--238. 

\bibitem{Mantel:1907}
W. Mantel: 
\newblock \emph{Vraagstuk {XXVIII}}, 
\newblock \emph{Wiskundige Opgaven met de Oplossingen}, {\bf 10}(1907), 60--61. 

\bibitem{Markstrom:2022}
K. Markstr\"{o}m: 
\newblock \emph{Covering designs}, available at \hspace*{\fill} \linebreak
\href{http://abel.math.umu.se/~klasm/Data/hypergraphs/coveringdesign.html}{abel.math.umu.se/{\textasciitilde}klasm/Data/hypergraphs/coveringdesign.html}.

\bibitem{Pikhurko:2011}
O. Pikhurko: 
\newblock \emph{The minimum size of $3$-graphs without a $4$-set spanning no or exactly three edges}, 
\newblock \emph{European J. Combin.}, {\bf 32}(2011), 1142--1155, 
\doi{10.1016/j.ejc.2011.03.006}.

\bibitem{Prachatice:1992}
\newblock \emph{Problems proposed at the problem session of the {P}rachatice {C}onference on {G}raph {T}heory},
\newblock \emph{Ann. Discrete Math.}, {\bf 51}(1992), 375--384,
\doi{10.1016/S0167-5060(08)70659-9}.

\bibitem{Razborov:2010}
A. Razborov: 
\newblock \emph{On $3$-hypergraphs with forbidden $4$-vertex configurations}, 
\newblock \emph{SIAM J. Discrete Math.}, {\bf 24}(2010), 946--963, 
\doi{10.1137/090747476}.

\bibitem{Sidorenko:1995}
A. Sidorenko: 
\newblock \emph{What we know and what we do not know about {T}ur\'{a}n numbers}, 
\newblock \emph{Graphs Combin.}, {\bf 11}(1995), 179--199, 
\doi{10.1007/BF01929486}.

\bibitem{Turan:1970}
P. Tur\'{a}n: 
\newblock \emph{Applications of graph theory to geometry and potential theory}, 
\newblock in \emph{Combinatorial Structures and Their Applications}, 
New York, Gordon and Breach, 1970, p.423--434. 

\bibitem{Turan:1941}
P. Tur\'{a}n: 
\newblock \emph{Egy gr\'{a}felm\'{e}leti sz\'{e}ls{\H{o}}\'{e}rt\'{e}kfeladatr\'{o}l}, 
\newblock \emph{Mat. Fiz. Lapok}, {\bf 48}(1941), 436--453. 

\end{thebibliography}
\end{document}